\documentclass[11pt,a4paper, oneside,reqno]{amsproc}
\usepackage{graphicx}
\usepackage{float}
\usepackage{pstricks}
\usepackage{amscd}
\usepackage{amsmath}
\usepackage{amsxtra}
\usepackage{hyperref}
\usepackage[T1]{fontenc}
\usepackage[utf8]{inputenc}
\usepackage{lipsum}
\usepackage{tikz}
\usepackage{amssymb, amsthm}
\usepackage{url}
\usepackage{enumerate}
\usepackage{mathrsfs}
\usepackage{epsfig}
\usepackage[active]{srcltx}

\def \dep{\mathsf{d}}

\usetikzlibrary{arrows}
\usetikzlibrary{calc}
\DeclareMathOperator*{\Oplus}{\ensuremath{\vcenter{\hbox{\scalebox{1.3}{$\oplus$}}}}}

\theoremstyle{plain}
\newtheorem{theorem}{Theorem}[section]
\newtheorem{c-theorem}{Construction theorem}[section]

\newtheorem{lemma}[theorem]{Lemma}
\newtheorem{proposition}[theorem]{Proposition}

\theoremstyle{definition}

\theoremstyle{remark}

\newlength{\struh}
\newlength{\textminustop}
\newcommand*{\child}[1]{\mathsf{Chi}(#1)}
\newcommand*{\childn}[2]{{\mathsf{Chi}}^{\langle#1\rangle}(#2)}

\newcommand*{\parentn}[2]{{\mathsf{par}}^{\langle#1\rangle}(#2)}

\newcommand*{\lambdab}{\boldsymbol\lambda}

\newcommand*{\parent}[1]{\mathsf{par}(#1)}

\newcommand*{\rootb}{{\mathsf{root}}}

\makeatletter
\@namedef{subjclassname@2020}{%
\textup{2020} Mathematics Subject Classification}
\makeatother

\newcommand{\ncom}{\newcommand}
\ncom{\bq}{\begin{equation}}
\ncom{\eq}{\end{equation}}
\ncom{\beqn}{\begin{eqnarray*}}
\ncom{\eeqn}{\end{eqnarray*}}
\ncom{\beq}{\begin{eqnarray}}
\ncom{\eeq}{\end{eqnarray}}
\ncom{\nno}{\nonumber}
\ncom{\rar}{\rightarrow}
\ncom{\Rar}{\Rightarrow}
\ncom{\noin}{\noindent}
\ncom{\bc}{\begin{centre}}
\ncom{\ec}{\end{centre}}
\ncom{\sz}{\scriptsize}
\ncom{\rf}{\ref}
\ncom{\sgm}{\sigma}
\ncom{\Sgm}{\Sigma}
\ncom{\dt}{\delta}
\ncom{\Dt}{Delta}
\ncom{\lmd}{\lambda}
\ncom{\Lmd}{\Lambda}
\ncom{\eps}{\epsilon}
\ncom{\pcc}{\stackrel{P}{>}}
\ncom{\dist}{{\rm\,dist}}
\ncom{\im}{{\rm Im\,}}
\ncom{\sgn}{{\rm sgn\,}}
\ncom{\ba}{\begin{array}}
\ncom{\ea}{\end{array}}
\ncom{\eop}{\hfill{{\rule{2.5mm}{2.5mm}}}}
\ncom{\eof}{\hfill{{\rule{1.5mm}{1.5mm}}}}
\ncom{\hone}{\mbox{\hspace{1em}}}
\ncom{\htwo}{\mbox{\hspace{2em}}}
\ncom{\hthree}{\mbox{\hspace{3em}}}
\ncom{\hfour}{\mbox{\hspace{4em}}}
\ncom{\hsev}{\mbox{\hspace{7em}}}
\ncom{\vone}{\vskip 2ex}
\ncom{\vtwo}{\vskip 4ex}
\ncom{\vonee}{\vskip 1.5ex}
\ncom{\vthree}{\vskip 6ex}
\ncom{\vfour}{\vspace*{8ex}}
\ncom{\norm}{\|\;\;\|}
\ncom{\integ}[4]{\int_{#1}^{#2}\,{#3}\,d{#4}}
\ncom{\inp}[2]{\langle{#1},\,{#2} \rangle}
\ncom{\Inp}[2]{\Langle{#1},\,{#2} \Langle}
\ncom{\vspan}[1]{{{\rm\,span}\#1 \}}}
\ncom{\dm}[1]{\displaystyle {#1}}

\begin{document}

\title[Unitary equivalence of balanced weighted shifts]{Unitary equivalence of balanced weighted shifts on rooted directed trees}

\author[S. Mandal]{Shubhankar Mandal}
\author[S. Trivedi]{Shailesh Trivedi}

\address{Department of Mathematics, Indian Institute of Technology Bhilai, 491002, India}

\email{shubhankarm@iitbhilai.ac.in}
  
\address{Department of Mathematics, Birla Institute of Technology and Science, Pilani, Pilani Campus, Vidya Vihar, Pilani, Rajasthan 333031, India}

 \email{shailesh.trivedi@pilani.bits-pilani.ac.in}


\thanks{The work of the first author is supported through SERB-CRG (CRG/2022/003058), whereas the work of the second author is supported partially through the SERB-SRG (SRG/2023/000641-G) and OPERA (FR/SCM/03-Nov-2022/MATH)}



\begin{abstract} 
We completely characterize non-periodic balanced weighted shifts $S_{\lambdab}$ on rooted directed trees under a very mild assumption that $S_{\lambdab}^{*n}S_{\lambdab}^n|_{\ker S_{\lambdab}^*}$ is invertible operator on $\ker S_{\lambdab}^*$ for all $n \in \mathbb N$. This generalizes the previously established unitary equivalences for Bergman and Dirichlet type shifts associated with locally finite rooted directed trees. We also give a counter example to justify that the criteria obtained for non-periodic balanced weighted shifts is not necessary for eventually periodic balanced weighted shifts.
\end{abstract}

\maketitle


\section{Introduction}

There has been an extensive study on the weighted shifts on directed trees since their introduction in operator theory \cite{JJS}, see for instance, \cite{ACJS2019, ACJS2019-2, ACT2020, B-D-P, BDPP, BJJS2017, CPT, CPT2017, CT, DPP, G, MS} and references therein. This short note is another contribution towards this study. In \cite{BDPP}, the authors introduced the notion of balanced weighted shifts on the rooted directed trees (see \cite{CPT2017} for the multivariable analogue of this notion) to study the reflexivity of these shifts. We carry forward this study and classify the balanced weighted shifts on rooted directed trees with non-periodic moments (see the paragraph prior to the Theorem \ref{Thm1.2}). To achieve this, we first model a balanced weighted shift $S_{\lambdab}$ as an operator $\mathscr M_z$ of multiplication by the coordinate function on a Hilbert space of vector-valued formal power series under a mild assumption that $S_{\lambdab}^{*n}S_{\lambdab}^n|_{\ker S_{\lambdab}^*}$ is invertible operator on $\ker S_{\lambdab}^*$ for all $n \in \mathbb N$ (see Lemma \ref{power-ker-inva}). This establishes that $S_{\lambdab}$ (with this assumption) is unitarily equivalent to an operator-valued unilateral weighted shift with invertible operator weights. Then we apply the unitary equivalence criteria for operator-valued unilateral weighted shifts with invertible operator weights to obtain the main result of this paper. To state the main result, we recall below a few basic facts pertaining to weighted shifts on directed trees. We refer the reader to \cite{JJS} for the definition of a weighted shift on a directed tree and related notions such as children, parent, branching vertices etc.

Let $\mathbb N$ and $\mathbb C$ denote the set of all non-negative integers and the field of complex numbers, respectively. For a set $X$, the cardinality of $X$ is denoted by $\text{Card}(X)$. Let $\mathscr T = (V, \mathcal E)$ be a rooted directed tree and the root vertex of $\mathscr T$ be denoted by $\rootb$. It follows from \cite[Proposition 2.1.2]{JJS} that for each $v \in V$ there exists a unique non-negative integer $\dep_v$ such that $v \in \childn{\dep_v}{\rootb}$. This unique non-negative integer $\dep_v$ is known as the {\it depth of} $v$. For $n \in \mathbb N$, the {\it $n$-th generation $\mathscr G_n$ of $\mathscr T$} is defined as:
\beq\label{gen-eq}
\mathscr G_n := \{v \in V : \dep_v = n\}.
\eeq
Let $V_\prec$ denote the set of all branching vertices of $\mathscr T$. The {\it branching index $k_\mathscr T$ of} $\mathscr T$ is defined as:  
\[k_\mathscr{T}:=\begin{cases}
 1+\sup\{\dep_v : v\in V_{\prec}\} & \text{if $V_{\prec}$ is non-empty},\\
 0 & \text{if $V_{\prec}$ is empty}.
\end{cases}
\]
Suppose that $S_\lambda$ is a bounded weighted shift on a rooted directed tree $\mathscr T = (V, \mathcal E)$. Then it follows from \cite[Eq. (2.2)]{CT} (see also \cite[Proposition 3.5.1(ii)]{JJS}) that 
\beq\label{ker-adjoint}
\ker S_{\lambdab}^*=[e_{\mathsf{root}}] \oplus \bigoplus_{v \in 
V_\prec}\Big(\ell^2(\mathsf{Chi}(v)) \ominus [\lambdab^v]\Big),
\eeq
where
$\lambdab^v : \mathsf{Chi}(v) \rar \mathbb C$ is defined as 
$\lambdab^v(u)=\lambda_u,$ and $[f]$ denotes the span of $\{f\}$. Since $\mathscr G_m \cap \mathscr G_n = \emptyset$ for $m \ne n$, \eqref{ker-adjoint} can be rewritten as 
\beq\label{ker-adjoint-gen}
\ker S_{\lambdab}^*=[e_{\mathsf{root}}] \oplus \bigoplus_{n \in \mathbb N}\, \bigoplus_{v \in V_\prec \cap \mathscr G_n} \Big(\ell^2(\mathsf{Chi}(v)) \ominus [\lambdab^v]\Big).
\eeq

Let $\mathscr T = (V, \mathcal E)$ be a countably infinite, rooted, leafless directed tree and $S_{\lambdab}$ be a bounded weighted shift on $\mathscr T$ with positive weights. Then following \cite{BDPP}, we say that $S_{\lambdab}$ is {\it balanced} if $\|S_{\lambdab} e_u\| = \|S_{\lambdab} e_v\|$ for all $u, v \in V$ such that $\dep_u = \dep_v$. Further, $S_{\lambdab}$ is said to be {\it locally power balanced} if $\|S^n_{\lambdab} e_u\| = \|S^n_{\lambdab} e_v\|$ for all integer $n \geqslant 1$ and all $u, v \in V$ such that $\parent u = \parent v$. It follows from \cite[Lemma 6.3]{BDPP} and \cite[Proposition 3.1.7]{JJS} that $S_{\lambdab}$ is balanced if and only if it is locally power balanced. Interesting examples of balanced weighted shifts appeared in \cite{CPT} known as Bergman and Dirichlet type shifts associated with a locally finite rooted directed tree. These examples are described as follows.

Let $\mathscr T = (V, \mathcal E)$ be a locally finite, countably infinite, rooted, leafless directed tree. For a real number $q \geqslant 1$, consider the weight system $\lambdab_q = \big\{\lambda_{u,q} : u \in V \setminus \{\rootb\}\big\}$ given  by
\beqn
\lambda_{u,q} = \frac{1}{\sqrt{\text{Card}(\child v)}} \sqrt{\frac{\dep_v+q}{\dep_v+1}} \ \text{ for }\ u \in \child v,\ v \in V.
\eeqn
Then the weighted shift $S_{\lambdab_q}$ on $\mathscr T$ associated with the weight system $\lambdab_q$ is called as Dirichlet type shift on $\mathscr T$. Note that $S_{\lambdab_q}$ is bounded and left-invertible. Similarly, if we define  
\beq\label{bergman}
\lambda_{u,q} = \frac{1}{\sqrt{\text{Card}(\child v)}} \sqrt{\frac{\dep_v+1}{\dep_v+q}} \ \text{ for }\ u \in \child v,\ v \in V,
\eeq
then the weighted shift associated with the weight system \eqref{bergman} is called as Bergman type shift on $\mathscr T$. It is also bounded and left-invertible. Note that if $q=1$, then both the Bergman and Dirichlet type weighted shifts are isometry.

For future reference, we rephrase a part of \cite[Theorem 6.4]{BDPP} as follows:

\begin{theorem}\label{BDPP}\cite[Theorem 6.4]{BDPP}
Let $S_{\lambdab}$ be a bounded weighted shift on a countably infinite, rooted, leafless directed tree $\mathscr T = (V, \mathcal E)$ with positive weights. Then the sequence $\{S^n_{\lambdab}(\ker S^*_{\lambdab})\}_{n \in \mathbb N}$ of subspaces of $\ell^2(V)$ is mutually orthogonal if and only if $S_{\lambdab}$ is balanced.
\end{theorem}   

\textit{Assumption:} Throughout the remainder of the paper, we assume that the directed tree $\mathscr T$ is always countably infinite and leafless. The weighted shift $S_{\lambdab}$ is a bounded weighted shift on $\mathscr T$ with positive weights.

Let $(a_n)_{n \in \mathbb N}$ be a sequence of complex numbers. We say that $(a_n)_{n \in \mathbb N}$ is {\it eventually periodic} if there exists $n_0 \in \mathbb N$ and an integer $k \geqslant 1$ such that $a_m = a_{m+k}$ for all $m \geqslant n_0$. If $n_0 = 0$, then $(a_n)_{n \in \mathbb N}$ is said to be periodic. Further, $(a_n)_{n \in \mathbb N}$ is said to be {\it non-periodic} if it is not eventually periodic. Note that if $n_0 = 0$ and $k=1$, then $(a_n)_{n \in \mathbb N}$ is a constant sequence. A balanced weighted shift $S_{\lambdab}$ on a rooted directed tree $\mathscr T = (V, \mathcal E)$ is said to be eventually periodic if the sequence $(c_n)_{n \in \mathbb N}$ is eventually periodic, where $c_n$ denotes the constant value of $\|S_{\lambdab} e_v\|$ for $v \in \mathscr G_n$. Similarly, a balanced weighted shift $S_{\lambdab}$ on a rooted directed tree $\mathscr T = (V, \mathcal E)$ is said to be non-periodic if the sequence $(c_n)_{n \in \mathbb N}$ is non-periodic. 

The Bergman and Dirichlet type weighted shifts are examples of non-periodic weighted shifts for $q>1$, whereas for $q=1$, these are periodic weighted shifts (in fact, $c_n = 1$ for all $n \in \mathbb N$).

We are now in a position to state the main result of this paper.

\begin{theorem}\label{Thm1.2}
Let $S_{\lambdab}$ and $\tilde S_{\lambdab}$ be balanced weighted shifts on rooted directed trees $\mathscr T = (V, \mathcal E)$ and $\tilde{\mathscr T} = (\tilde V, \tilde{\mathcal E})$, respectively with $E := \ker S_{\lambdab}^*$ and $\tilde E := \ker \tilde S_{\lambdab}^*$. Suppose that $S_{\lambdab}^{*n}S_{\lambdab}^n|_E$ and $\tilde S_{\lambdab}^{*n} \tilde S_{\lambdab}^n|_{\tilde E}$ are invertible operators on $E$ and $\tilde E$, respectively for all $n \in \mathbb N$. If $S_{\lambdab}$ and $\tilde S_{\lambdab}$ are non-periodic, then $S_{\lambdab}$ is unitarily equivalent to $\tilde S_{\lambdab}$ if and only if
\[c_n = \tilde c_n \text{ and Card}(\mathscr G_n) = \text{Card}(\tilde{\mathscr G}_n)\ \text{ for all }\ n \in \mathbb N,\]
where $c_n$ $($resp. $\tilde c_n)$ denotes the constant value of $\|S_{\lambdab} e_v\|$ $($resp. $\|\tilde S_{\lambdab} e_{\tilde v}\|)$ for $v \in \mathscr G_n\, ($resp. $\tilde v \in \tilde{\mathscr G}_n)$.
\end{theorem}

Theorem \ref{Thm1.2} generalizes \cite[Theorem 2.4]{CPT} which provides the unitary equivalence of Dirichlet type (and consequently Bergman type) weighted shifts associated with locally finite rooted directed trees. Further, it is worth emphasizing that the assumption that $S_{\lambdab}^{*n}S_{\lambdab}^n|_{\ker S_{\lambdab}^*}$ is invertible operator on $\ker S_{\lambdab}^*$ for all $n \in \mathbb N$, is very mild. In fact, in view of Lemma \ref{power-ker-inva}, it is weaker than the left-invertibility of $S_{\lambdab}$.

\section{Preparatory results}

In this section, we collect a couple of lemmas which are crucial for the proof of Theorem \ref{Thm1.2}. We proceed with the following lemma which describes a sort of multiplicative structure of the moments of a balanced weighted shift on a rooted directed tree.

\begin{lemma}\label{lemma-cor}
Let $S_{\lambdab}$ be a balanced weighted shift on a rooted directed tree $\mathscr T = (V, \mathcal E)$. Let $c_m$ denote the constant value of $\|S_{\lambdab} e_v\|$ for $v \in \mathscr G_m$, $m \in \mathbb N$. Then for each integer $n \geqslant 1$, we have 
\beqn\label{root-v-eq}
\|S^n_{\lambdab} e_v\| = c_{\dep v}\cdots c_{\dep v+n-1} \ \text{ for all }\ v \in V.
\eeqn
\end{lemma}

\begin{proof}
Let $v \in V$. It follows from \cite[Lemma 19]{DPP} that for all integer $n \geqslant 1$ and $u \in \childn{n-1}{v}$, 
\beqn
\|S^{\dep_v +n}_\lambda e_{\rootb}\| &=& \Big(\prod_{j=0}^{n-1} \|S_\lambda e_{\parentn{j}{u}}\| \Big) \|S_\lambda e_{\parent{v}}\| \cdots \|S_\lambda e_{\rootb}\| \\
&=& \|S^n_\lambda e_{v}\| \|S^{\dep_{v}}_\lambda e_{\rootb}\|.
\eeqn
This gives that
\beqn
\|S^n_\lambda e_{v}\| = \frac{\|S^{\dep_v +n}_\lambda e_{\rootb}\|}{\|S^{\dep_{v}}_\lambda e_{\rootb}\|} = \frac{c_0 \cdots c_{\dep v+n-1}}{c_0 \cdots c_{\dep v-1}} = c_{\dep v}\cdots c_{\dep v+n-1}
\eeqn
for all $v \in V$ and for all integer $n \geqslant 1$.
\end{proof}

\begin{lemma}\label{power-ker-inva}
Let $S_{\lambdab}$ be a balanced weighted shift on a rooted directed tree $\mathscr T = (V, \mathcal E)$ and $E := \ker S_{\lambdab}^*$. Then $S_{\lambdab}^{*n}S_{\lambdab}^n(E)$ is a dense subspace of $E$ for all $n \in \mathbb N$. Consequently, $S_{\lambdab}^{*n}S_{\lambdab}^n|_E$ is a positive operator on $E$ for all $n \in \mathbb N$. Moreover, if any of the following conditions is satisfied, then $S_{\lambdab}^{*n}S_{\lambdab}^n|_E$ is an invertible operator on $E$ for all $n \in \mathbb N$:
\begin{itemize}
\item[(i)] $E$ is finite dimensional.
\item[(ii)] $S_{\lambdab}$ is left-invertible.
\item[(iii)] $\mathscr T$ is of finite branching index.
\end{itemize}  
\end{lemma}

\begin{proof}
We first show that $S_{\lambdab}^{*n}S_{\lambdab}^n(E) \subseteq E$ for all $n \in \mathbb N$. Since $S_{\lambdab}$ is balanced, it follows from Theorem \ref{BDPP} that for all $m, n \in \mathbb N$, $S^n_{\lambdab}(E) \perp S^m_{\lambdab}(E)$ if $m \ne n$. Thus, for any $x, y \in E$ and $n \in \mathbb N$, we get 
\[\inp{S_{\lambdab}^{*n}S_{\lambdab}^n x}{S^m_{\lambdab} y} = 0 \ \text{ for all integer }\ m \geq 1.\]
That is, $\bigoplus_{m=1}^\infty S_{\lambdab}^m(E) \subseteq S_{\lambdab}^{*n}S_{\lambdab}^n(E)^\perp$. Note that every bounded weighted shift on a rooted directed tree has the wandering subspace property (see \cite[Proposition 1.3.4]{CPT2017}). Therefore, we obtain that 
\[\ell^2(V) = \bigoplus_{n=0}^\infty S_{\lambdab}^n(E).\]
This, together with above, gives us $S_{\lambdab}^{*n}S_{\lambdab}^n(E) \subseteq E$ for all $n \in \mathbb N$. Now fix $n \in \mathbb N$ and let $x \in E \ominus S_{\lambdab}^{*n}S_{\lambdab}^n(E)$. Then 
\[\inp{S_{\lambdab}^{*n}S_{\lambdab}^n y}{x} = 0 \ \text{ for all }\ y \in E.\]
In particular, $\|S_{\lambdab}^n x\| = 0$, which implies that $x=0$. Thus, $S_{\lambdab}^{*n}S_{\lambdab}^n(E)$ is a dense subspace of $E$. That $S_{\lambdab}^{*n}S_{\lambdab}^n|_E$ is a positive operator on $E$ for all $n \in \mathbb N$, is obvious. 

To see the moreover part, fix $n \in \mathbb N$. Note that if $E$ is finite dimensional, then the invertibility of $S_{\lambdab}^{*n}S_{\lambdab}^n|_E$ follows from the rank-nullity theorem. Assume that $S_{\lambdab}$ is left-invertible. Then there exists $\alpha(n) > 0$ such that $\|S_{\lambdab}^n e_v\| \geqslant \alpha(n)$ for all $v \in V$. Let $S_{\lambdab}^{*n}S_{\lambdab}^n x_m \to y$ as $m \to \infty$, where $x_m = \sum_{v \in V} x_m(v) e_v \in E$ for all $m \geqslant 1$. Then for all $m,k \geqslant 1$, we have
\[\|S_{\lambdab}^{*n}S_{\lambdab}^n (x_m-x_k)\|^2 = \sum_{v \in V}|x_m(v) - x_k(v)|^2 \|S_{\lambdab}^n e_v\|^4 \geqslant \alpha(n)^4 \|x_m - x_k\|^2. \]
Thus, $(x_m)_{m \geqslant 1}$ is a Cauchy sequence in $E$. Let $x_m \to x \in E$ as $m \to \infty$. Then $y = S_{\lambdab}^{*n}S_{\lambdab}^n x$. This shows that $S_{\lambdab}^{*n}S_{\lambdab}^n(E)$ is a closed subspace of $E$, and hence by preceding paragraph, $S_{\lambdab}^{*n}S_{\lambdab}^n(E) = E$. Consequently, $S_{\lambdab}^{*n}S_{\lambdab}^n|_E$ is an invertible operator on $E$.

Finally, suppose that the branching index $k_{\mathscr T}$ of $\mathscr T$ is finite. Fix $n \in \mathbb N$. Let $c_m$ denote the constant value of $\|S_{\lambdab} e_v\|$ for $v \in \mathscr G_m$, $m \in \mathbb N$. Then $c := \min\{c_0, \ldots, c_{k_{\mathscr T}}\} > 0$. Let $x \in E$. Then it follows from \eqref{ker-adjoint} that
\[x = a e_{\mathsf{root}} + \sum_{v \in V_\prec} b_v\ \text{ for some }\ a \in \mathbb C \ \text{ and }\ b_v \in \ell^2(\mathsf{Chi}(v)) \ominus [\lambdab^v].\]
In view of Lemma \ref{lemma-cor}, we get that 
\[S_{\lambdab}^{*n}S_{\lambdab}^n b_v = \|S_{\lambdab}^n e_w\|^2 b_v = c_{\dep v+1}^2\cdots c_{\dep v+n}^2 b_v,\]
where $w$ is any element of $\mathsf{Chi}(v)$. Consider
\[y = \frac{a}{c_0^2\cdots c_{n-1}^2} e_{\mathsf{root}} + \sum_{v \in V_\prec} \frac{b_v}{c_{\dep v+1}^2\cdots c_{\dep v+n}^2}.\]
Then 
\beqn
\|y\|^2 &=& \frac{|a|^2}{(c_0^2\cdots c_{n-1}^2)^2} + \sum_{v \in V_\prec} \frac{\|b_v\|^2}{(c_{\dep v+1}^2\cdots c_{\dep v+n}^2)^2}\\
&\leqslant& \frac{|a|^2}{c^{4n}} + \sum_{v \in V_\prec} \frac{\|b_v\|^2}{c^{4n}} = \frac{1}{c^{4n}} \|x\|^2 < \infty. 
\eeqn
Further, it can be easily verified that $y \in E$ and $S_{\lambdab}^{*n}S_{\lambdab}^n y = x$. Hence, $S_{\lambdab}^{*n}S_{\lambdab}^n|_E$ is an invertible operator on $E$.
\end{proof}

\section{Classification}

In this section, we classify all the non-periodic balanced weighted shifts $S_{\lambdab}$ on rooted directed trees under a mild assumption that $S_{\lambdab}^{*n}S_{\lambdab}^n|_{\ker S_{\lambdab}^*}$ is invertible operator on $\ker S_{\lambdab}^*$ for all $n \in \mathbb N$. In order to get the desired classification, we first model $S_{\lambdab}$ as an operator of multiplication by the coordinate function $z$ on the Hilbert space of $\ker S_{\lambdab}^*$-valued formal power series associated with the sequence $(S_{\lambdab}^{*n}S_{\lambdab}^n|_{\ker S_{\lambdab}^*})_{n \in \mathbb N}$ of invertible operators on $\ker S_{\lambdab}^*$. For the sake of convenience, we briefly recall a few basic things about a Hilbert space of vector-valued formal power series. The reader is referred to \cite{GKT, GKT1, GKT2024} for a detailed study on such Hilbert spaces. 

Let $E$ be a complex separable Hilbert space. An $E$-valued formal power series refers to a series of the form $\sum_{n=0}^\infty x_n z^n,\ x_n \in E$, without any concern for its convergence at any point $z \in \mathbb C$. Let $\mathscr B=\{B_n : n \in \mathbb N\}$ be a sequence of bounded linear operators on $E$ and consider
\beqn
\mathcal H^2_E(\mathscr B) := \bigg\{\sum_{n=0}^\infty x_n z^n : \sum_{n=0}^\infty \|B_n x_n\|^2 < \infty \bigg\}
\eeqn
equipped with the following inner product: 
\beqn
\inp{f}{g}_{\mathcal H^2_E(\mathscr B)} :=  \sum_{n=0}^\infty \inp{B_n x_n}{B_n y_n}_E, \quad f = \sum_{n=0}^\infty x_n z^n,\ g = \sum_{n=0}^\infty y_n z^n \in \mathcal H^2_E(\mathscr B).
\eeqn     
Then $\mathcal H^2_E(\mathscr B)$ is a Hilbert space known as the {\it Hilbert space of $E$-valued formal power series}. Note that $E$-valued polynomials are dense in $\mathcal H^2_E(\mathscr B)$. Thus, the operator $\mathscr M_z$ of multiplication by the coordinate function $z$ is a densely defined operator in $\mathcal H^2_E(\mathscr B)$ which may not be bounded in general. It turns out that a balanced weighted shift on a rooted directed tree can be realized as (a bounded) $\mathscr M_z$ on $\mathcal H^2_E(\mathscr B)$ for some $E$ and $\mathscr B$, which is illustrated in the following proposition.

\begin{proposition}\label{model-prop}
Let $S_{\lambdab}$ be a balanced weighted  shift on a rooted directed tree $\mathscr T = (V, \mathcal E)$ and $E := \ker S_{\lambdab}^*$. Then $S_{\lambdab}$ is unitarily equivalent to $\mathscr M_z$ on $\mathcal H^2_E(\mathscr B)$ with $\mathscr B=\big\{(S_{\lambdab}^{*n}S_{\lambdab}^n|_E)^{1/2} : n \in \mathbb N\big\}$. 
\end{proposition}

\begin{proof}
We first show that $\mathscr M_z$ is bounded on $\mathcal H^2_E(\mathscr B)$. To this end, for any $f = \sum_{n=0}^\infty x_n z^n \in \mathcal H^2_E(\mathscr B)$, we have
\beqn
\|\mathscr M_z f\|^2 &=& \Big\|\sum_{n=0}^\infty x_n z^{n+1}\Big\|^2 = \sum_{n=0}^\infty \big\langle S_{\lambdab}^{*n+1}S_{\lambdab}^{n+1} x_n,\ x_n\big\rangle = \sum_{n=0}^\infty \|S_{\lambdab}^{n+1} x_n\|^2\\
&\leqslant& \|S_{\lambdab}\|^2 \sum_{n=0}^\infty \|S_{\lambdab}^{n} x_n\|^2 = \|S_{\lambdab}\|^2 \|f\|^2 < \infty.
\eeqn
Thus, $\mathscr M_z$ is bounded. 

Since $S_{\lambdab}$ is a balanced weighted shift and every bounded weighted shift on a rooted directed tree has the wandering subspace property (see \cite[Proposition 1.3.4]{CPT2017}), from Theorem \ref{BDPP}, we get 
\beqn
\ell^2(V) = \bigoplus_{n=0}^\infty S_{\lambdab}^n(E).
\eeqn
Now define $U : \ell^2(V) \to \mathcal H^2_E(\mathscr B)$ as
\beqn
U\Big(\Oplus_{n=0}^\infty S_{\lambdab}^n x_n\Big) = \sum_{n=0}^\infty x_n z^n, \quad \Oplus_{n=0}^\infty S_{\lambdab}^n x_n \in \ell^2(V).
\eeqn
It can be easily verified that $U$ is unitary and $US_{\lambdab} = \mathscr M_z U$. 
\end{proof}

It is very natural to ask when the elements of $\mathcal H^2_E(\mathscr B)$ define $E$-valued holomorphic functions. The answer lies in determining the set of all bounded point evaluations on $\mathcal H^2_E(\mathscr B)$. Note that $E$-valued polynomials are dense in $\mathcal H^2_E(\mathscr B)$. Therefore, for $w \in \mathbb C$, let us consider the evaluation map $\mathscr E_w$ given by $\mathscr E_w p = p(w)$ for all polynomials $p \in \mathcal H^2_E(\mathscr B)$. We say that $w \in \mathbb C$ is a bounded point evaluation on $\mathcal H^2_E(\mathscr B)$ if the evaluation map $\mathscr E_w$ extends to a continuous linear map from $\mathcal H^2_E(\mathscr B)$ onto $E$. The continuous extension of $\mathscr E_w$, by an abuse of notation, is denoted by $\mathscr E_w$ itself. Let $\Omega$ denote the set of all bounded point evaluations on $\mathcal H^2_E(\mathscr B)$. It turns out that if $S_{\lambdab}$ is left-invertible, then the open disc $\mathbb D\big(0, r(S'_{\lambdab})^{-1}\big)$ is contained in $\Omega$, where $r(T)$ denotes the spectral radius of a bounded linear operator $T$ and $T' := T(T^*T)^{-1}$ denotes the Cauchy dual of a left-invertible operator $T$. Although this inclusion follows from \cite{S2001} but here we provide a direct proof.
 
\begin{proposition}\label{bpe-prop}
Let $S_{\lambdab}$ be a balanced weighted shift on a rooted directed tree $\mathscr T = (V, \mathcal E)$ and $E := \ker S_{\lambdab}^*$. If $S_{\lambdab}$ is left-invertible, then the set $\Omega$ of all bounded point evaluations on $\mathcal H^2_E(\mathscr B)$ contains $\mathbb D\big(0, r(S'_{\lambdab})^{-1}\big)$.
\end{proposition}
  
\begin{proof}
Suppose that $S_{\lambdab}$ is left-invertible. Let $f = \sum_{n=0}^\infty x_n z^n \in \mathcal H^2_E(\mathscr B)$ and $w \in \mathbb D\big(0, r(S'_{\lambdab})^{-1}\big)$. Then
\beqn
\|f(w)\| &\leqslant& \sum_{n=0}^\infty |w|^n \|x_n\| = \sum_{n=0}^\infty |w|^n \|S_{\lambdab}'^{*n} S_{\lambdab}^n x_n\|\\ 
&\leqslant& \left(\sum_{n=0}^\infty |w|^{2n} \|S_{\lambdab}'^{n}\|^2\right)^\frac{1}{2} \left(\sum_{n=0}^\infty \|S_{\lambdab}^n x_n\|^2\right)^\frac{1}{2} = \eta(w) \|f\|,
\eeqn
where $\eta(w) = \left(\sum_{n=0}^\infty |w|^{2n} \|S_{\lambdab}'^{n}\|^2\right)^\frac{1}{2} < \infty$ by virtue of the fact that $|w| < 1/r(S_{\lambdab}')$. This shows that the evaluation map $\mathscr E_w : \mathcal H^2_E(\mathscr B) \to E$ is continuous. Consequently, $\mathbb D\big(0, r(S'_{\lambdab})^{-1}\big) \subseteq \Omega$.
\end{proof}

If $\Omega$ has non-empty interior $\Omega^\circ$, it follows from the general theory \cite{PR} that $\mathcal H^2_E(\mathscr B)$ is a reproducing kernel Hilbert space of $E$-valued holomorphic functions defined on $\Omega^\circ$. In addition, if $S_{\lambdab}^{*n}S_{\lambdab}^n|_E$ is invertible operator on $E$ for all $n \in \mathbb N$, then from \cite[eq.(11)]{GKT1}, the reproducing kernel $\kappa : \Omega^\circ \times \Omega^\circ \to \mathcal B(E)$ is given by 
$$\kappa(z,w) = \mathscr E_z \mathscr E_w^* = \sum_{n=0}^\infty (S_{\lambdab}^{*n}S_{\lambdab}^n|_E)^{-1} z^n \bar{w}^n, \quad z,w \in \Omega^\circ.$$

The following proposition shows that the above expression of $\kappa$ can be further refined to get more explicit form. It also generalizes the analytic models obtained in \cite{CPT} for the Bergman and Dirichlet type weighted shifts associated with locally finite rooted directed trees.

\begin{proposition}\label{analytic-model}
Let $S_{\lambdab}$ be a balanced weighted shift on a rooted directed tree $\mathscr T = (V, \mathcal E)$ and $E := \ker S_{\lambdab}^*$. Suppose that $S_{\lambdab}^{*n}S_{\lambdab}^n|_E$ is invertible operator on $E$ for all $n \in \mathbb N$. If $\Omega^\circ$ is non-empty, then $\mathcal H^2_E(\mathscr B)$ is a reproducing kernel Hilbert space of $E$-valued holomorphic functions defined on $\Omega^\circ$ and the reproducing kernel $\kappa : \Omega^\circ \times \Omega^\circ \to \mathcal B(E)$ is given by
\beqn
\kappa(z,w) = \sum_{n=0}^\infty \frac{z^n \overline{w}^n}{c_0^2 \cdots c_{n-1}^2} P_{[e_{\mathsf{root}}]} + \sum_{v \in V_\prec} \sum_{n=0}^\infty \frac{z^n \overline{w}^n}{c_{\dep v+1}^2 \cdots c_{\dep v+n}^2} P_{\ell^2(\mathsf{Chi}(v)) \ominus [\lambdab^v]},
\eeqn
for all $z, w \in \Omega^\circ$, where $c_n$ denotes the constant value of $\|S_{\lambdab} e_v\|$ for $v \in \mathscr G_n$, and $P_M$ denotes the orthogonal projection onto a closed subspace $M$ of a Hilbert space $H$.
\end{proposition}

\begin{proof}
In view of the preceding discussion, we only need to show that $\kappa$ has the above mentioned form. Let $x \in E$. Then it follows from \eqref{ker-adjoint} that 
\beqn
x = P_{[e_{\mathsf{root}}]}(x) + \sum_{v \in V_\prec}P_{\ell^2(\mathsf{Chi}(v)) \ominus [\lambdab^v]}(x).
\eeqn
Hence, following the arguments of the proof of Lemma \ref{power-ker-inva}(iii) and using Lemma \ref{lemma-cor}, we get 
\beqn
S_{\lambdab}^{*n}S_{\lambdab}^n x = c_0^2 \cdots c_{n-1}^2 P_{[e_{\mathsf{root}}]}(x) + \sum_{v \in V_\prec} c_{\dep v+1}^2 \cdots c_{\dep v+n}^2 P_{\ell^2(\mathsf{Chi}(v)) \ominus [\lambdab^v]}(x)
\eeqn
for all $n \in \mathbb N$. This, in turn, yields that
\beqn
\kappa(z,w)x &=& \sum_{n=0}^\infty (S_{\lambdab}^{*n}S_{\lambdab}^n|_E)^{-1} x\, z^n \bar{w}^n = \sum_{n=0}^\infty \frac{z^n \overline{w}^n}{c_0^2 \cdots c_{n-1}^2} P_{[e_{\mathsf{root}}]}(x)\\ 
&+& \sum_{v \in V_\prec} \sum_{n=0}^\infty \frac{z^n \overline{w}^n}{c_{\dep v+1}^2 \cdots c_{\dep v+n}^2} P_{\ell^2(\mathsf{Chi}(v)) \ominus [\lambdab^v]}(x)
\eeqn
for all $z, w \in \Omega^\circ$, which completes the proof.
\end{proof}

The following lemma is a particular case of \cite[Corollary 2.5]{GKT2024} and is crucial for the classification of non-periodic balanced weighted shifts.

\begin{lemma}\label{lem-3.5}
Let $S_{\lambdab}$ and $\tilde S_{\lambdab}$ be balanced weighted shifts on rooted directed trees $\mathscr T = (V, \mathcal E)$ and $\tilde{\mathscr T} = (\tilde V, \tilde{\mathcal E})$, respectively with $E := \ker S_{\lambdab}^*$ and $\tilde E := \ker \tilde S_{\lambdab}^*$. Suppose that $S_{\lambdab}^{*n}S_{\lambdab}^n|_E$ and $\tilde S_{\lambdab}^{*n} \tilde S_{\lambdab}^n|_{\tilde E}$ are invertible operators on $E$ and $\tilde E$, respectively for all $n \in \mathbb N$. Then $S_{\lambdab}$ is unitarily equivalent to $\tilde S_{\lambdab}$ if and only if there exists a unitary $U: E \to \tilde E$ such that
\beq\label{unit-equiv}
U S_{\lambdab}^{*n} S_{\lambdab}^n|_E = \tilde S_{\lambdab}^{*n} \tilde S_{\lambdab}^n|_{\tilde E}\, U \ \text{ for all }\ n \in \mathbb N.
\eeq
\end{lemma}

The next lemma is an intermediate step of the proof of Theorem \ref{Thm1.2} for which we need the sequence $(W_n)_{n \in \mathbb N}$ (resp. $(\tilde W_n)_{n \in \mathbb N}$) of subspaces of $\ker S_{\lambdab}^*$ (resp. $\ker \tilde S_{\lambdab}^*$) defined as follows:
\beq\label{W-n}
\begin{split}
W_0 &:= [e_{\mathsf{root}}]\ \text{ and }\ W_n := \bigoplus_{v \in 
V_\prec \cap \mathscr G_{n-1}}\Big(\ell^2(\mathsf{Chi}(v)) \ominus [\lambdab^v]\Big),\ n \geqslant 1,\\
\tilde W_0 &:= [e_{\tilde{\mathsf{root}}}]\ \text{ and }\ \tilde W_n := \bigoplus_{\tilde v \in 
\tilde V_\prec \cap \tilde{\mathscr G}_{n-1}}\Big(\ell^2(\mathsf{Chi}(\tilde v)) \ominus [\lambdab^{\tilde v}]\Big),\ n \geqslant 1.
\end{split}
\eeq
Then it follows from \eqref{ker-adjoint-gen} that 
\beq\label{ker-W-n}
\ker S_{\lambdab}^* = \bigoplus_{n \in \mathbb N} W_n\ \text{ and }\ \ker \tilde S_{\lambdab}^* = \bigoplus_{n \in \mathbb N} \tilde W_n.
\eeq

\begin{lemma}\label{lem-3.6}
Let $S_{\lambdab}$ and $\tilde S_{\lambdab}$ be balanced weighted shifts on rooted directed trees $\mathscr T = (V, \mathcal E)$ and $\tilde{\mathscr T} = (\tilde V, \tilde{\mathcal E})$, respectively with $E := \ker S_{\lambdab}^*$ and $\tilde E := \ker \tilde S_{\lambdab}^*$. Suppose that $S_{\lambdab}^{*n}S_{\lambdab}^n|_E$ and $\tilde S_{\lambdab}^{*n} \tilde S_{\lambdab}^n|_{\tilde E}$ are invertible operators on $E$ and $\tilde E$, respectively for all $n \in \mathbb N$. Let $U: E \to \tilde E$ be a unitary satisfying \eqref{unit-equiv}. If $S_{\lambdab}$ and $\tilde S_{\lambdab}$ are non-periodic, then for each $k \in \mathbb N$, $U(W_k) = \tilde W_m$ for a unique $m \in \mathbb N$, where $W_n$ and $\tilde W_n$ are given by \eqref{W-n}.
\end{lemma}

\begin{proof}
Suppose that $S_{\lambdab}$ and $\tilde S_{\lambdab}$ are non-periodic. Fix $k \in \mathbb N$ and let $f \in W_k$. Since $Uf \in \tilde E$, let 
\[Uf = \sum_{m \in \mathbb N} g_m, \quad g_m \in \tilde W_m.\]
Let $c_n$ $($resp. $\tilde c_n)$ denote the constant value of $\|S_{\lambdab} e_v\|$ $($resp. $\|\tilde S_{\lambdab} e_{\tilde v}\|)$ for $v \in \mathscr G_n\, ($resp. $\tilde v \in \tilde{\mathscr G}_n)$. Suppose that $g_{m_1}$ and $g_{m_2}$ are non-zero for some distinct $m_1, m_2 \in \mathbb N$. Then by \eqref{unit-equiv} together with Lemma \ref{lemma-cor}, we get
\beqn
c_k^2 \cdots c_{k+n-1}^2 \Big(\sum_{m \in \mathbb N} g_m\Big) &=& c_k^2 \cdots c_{k+n-1}^2 U f = U S_{\lambdab}^{*n} S_{\lambdab}^n f
= \tilde S_{\lambdab}^{*n} \tilde S_{\lambdab}^n U f\\ 
&=& \sum_{m \in \mathbb N} \tilde c_m^2 \cdots \tilde c_{m+n-1}^2 g_m,
\eeqn
for all $n \in \mathbb N$. Taking the inner product on both sides with $g_{m_1}$ and then by $g_{m_2}$, we get
\[c_{k+j} = \tilde c_{m_1+j} = \tilde c_{m_2+j}\ \text{ for all }\ j \in \mathbb N.\]
Assume without loss of generality that $m_1 < m_2$. Then the above becomes
\[\tilde c_{m_1+j} = \tilde c_{(m_2-m_1)+m_1+j}\ \text{ for all }\ j \in \mathbb N.\]
This contradicts the fact that $\tilde S_{\lambdab}$ is non-periodic. Hence, there exists a unique $m \in \mathbb N$ such that $Uf = g_m$. This shows that $U(W_k) \subseteq \tilde W_m$. Now suppose that there exists $g \in \tilde W_m \setminus U(W_k)$. Then there exist $k_1 \in \mathbb N$ with $k_1 \ne k$ and $f_1 \in W_{k_1}$ such that $Uf_1 = g$. Again, applying \eqref{unit-equiv} and using Lemma \ref{lemma-cor}, we obtain
\beqn
c_{k_1}^2 \cdots c_{k_1+n-1}^2 g &=& c_{k_1}^2 \cdots c_{k_1+n-1}^2 U f_1 = U S_{\lambdab}^{*n} S_{\lambdab}^n f_1
= \tilde S_{\lambdab}^{*n} \tilde S_{\lambdab}^n U f_1\\ 
&=& \tilde c_m^2 \cdots \tilde c_{m+n-1}^2 g,
\eeqn
for all $n \in \mathbb N$. Comparing the coefficients on both sides, we get
$c_{k_1+j} = \tilde c_{m+j}\ \text{ for all }\ j \in \mathbb N$. Similar arguments for $Uf = g_m$ yields that $c_{k+j} = \tilde c_{m+j}\ \text{ for all }\ j \in \mathbb N$. Combining both gives that 
\[c_{k_1+j} = c_{k+j}\ \text{ for all }\ j \in \mathbb N.\]
This contradicts the fact that $S_{\lambdab}$ is non-periodic. Thus, $U(W_k) = \tilde W_m$ for a unique $m \in \mathbb N$, completing the proof.
\end{proof}

We now prove the main result of this article, which is Theorem \ref{Thm1.2}. It is restated here for the sake of convenience.

\begin{theorem}\label{Thm3.5}
Let $S_{\lambdab}$ and $\tilde S_{\lambdab}$ be balanced weighted shifts on rooted directed trees $\mathscr T = (V, \mathcal E)$ and $\tilde{\mathscr T} = (\tilde V, \tilde{\mathcal E})$, respectively with $E := \ker S_{\lambdab}^*$ and $\tilde E := \ker \tilde S_{\lambdab}^*$. Suppose that $S_{\lambdab}^{*n}S_{\lambdab}^n|_E$ and $\tilde S_{\lambdab}^{*n} \tilde S_{\lambdab}^n|_{\tilde E}$ are invertible operators on $E$ and $\tilde E$, respectively for all $n \in \mathbb N$. If $S_{\lambdab}$ and $\tilde S_{\lambdab}$ are non-periodic, then $S_{\lambdab}$ is unitarily equivalent to $\tilde S_{\lambdab}$ if and only if
\[c_n = \tilde c_n \text{ and Card}(\mathscr G_n) = \text{Card}(\tilde{\mathscr G}_n)\ \text{ for all }\ n \in \mathbb N,\]
where $c_n$ $($resp. $\tilde c_n)$ denotes the constant value of $\|S_{\lambdab} e_v\|$ $($resp. $\|\tilde S_{\lambdab} e_{\tilde v}\|)$ for $v \in \mathscr G_n\, ($resp. $\tilde v \in \tilde{\mathscr G}_n)$.
\end{theorem}

\begin{proof}
Let $S_{\lambdab}$ and $\tilde S_{\lambdab}$ be non-periodic. Suppose that $S_{\lambdab}$ is unitarily equivalent to $\tilde S_{\lambdab}$ and  $U : E \to \tilde E$ is a unitary satisfying \eqref{unit-equiv}. We claim that $U(W_k) = \tilde W_k$ for all $k \in \mathbb N$, where $W_k$ and $\tilde W_k$ are given by \eqref{W-n}. We prove the claim by mathematical induction. To this end, note that by Lemma \ref{lem-3.6}, $U(W_0) = \tilde W_m$ for some $m \in \mathbb N$. On contrary, assume that $m \geqslant 1$. Then $U e_{\rootb} = g$ for some $g \in \tilde W_m$. Applying \eqref{unit-equiv} together with Lemma \ref{lemma-cor}, we have
\beqn
c_0^2\cdots c_{n-1}^2 g &=& c_0^2\cdots c_{n-1}^2 U e_{\rootb} = U S_{\lambdab}^{*n} S_{\lambdab}^n e_{\mathsf{root}} = \tilde S_{\lambdab}^{*n} \tilde S_{\lambdab}^n U e_{\mathsf{root}}\\ 
&=& \tilde c_m^2 \cdots \tilde c_{m+n-1}^2 g
\eeqn
for all $n \in \mathbb N$. Comparing the coefficients on both sides, we get
\beq\label{01}
c_j = \tilde c_{m+j}\ \text{ for all }\ j \in \mathbb N. 
\eeq
Also by Lemma \ref{lem-3.6}, there exists $k \in \mathbb N$ such that $U(W_k) = \tilde W_0$. Note that $k \geqslant 1$. Then $Uf = e_{\tilde{\mathsf{root}}}$ for some $f \in W_k$. Again applying \eqref{unit-equiv} together with Lemma \ref{lemma-cor}, we get
\beqn
c_k^2\cdots c_{k+n-1}^2 e_{\tilde{\mathsf{root}}} &=& c_k^2\cdots c_{k+n-1}^2 Uf  = U S_{\lambdab}^{*n} S_{\lambdab}^n f = \tilde S_{\lambdab}^{*n} \tilde S_{\lambdab}^n U f\\ 
&=& \tilde c_0^2 \cdots \tilde c_{n-1}^2 e_{\tilde{\mathsf{root}}}
\eeqn
for all $n \in \mathbb N$. Comparing the coefficients on both sides, we get
\beq\label{02}
c_{k+j} = \tilde c_{j}\ \text{ for all }\ j \in \mathbb N.
\eeq
Combining \eqref{01} and \eqref{02}, we obtain that
\beqn
c_j = \tilde c_{m+j} = c_{m+k+j}\ \text{ for all }\ j \in \mathbb N.
\eeqn
Since $m+k \geqslant 2 > 0$, it follows that $S_{\lambdab}$ is periodic. This is a contradiction. Hence, $U(W_0) = \tilde W_0$, and consequently, by \eqref{01}, we get
\beqn
c_j = \tilde c_j\ \text{ for all }\ j \in \mathbb N.
\eeqn
Suppose that for some $k \in \mathbb N$, $U(W_l) = \tilde W_l$ for all $l \in \{0, \ldots, k\}$. On contrary, assume that $U(W_{k+1}) = \tilde W_m$ for some $m \geqslant k+2$. Then there exists $p \in \mathbb N$ such that $U(W_p) = \tilde W_{k+1}$. Note that $p \geqslant k+2$. Similar arguments as above show that
\beqn
c_{k+1+j} = \tilde c_{m+j} \ \text{ and }\ c_{p+j} = \tilde c_{k+1+j}\ \text{ for all }\ j \in \mathbb N.
\eeqn
This, in turn, gives that
\beqn
c_{k+1+j} = \tilde c_{m+j} = \tilde c_{(m-k-1)+k+1+j} = c_{(m-k-1)+p+j}\ \text{ for all }\ j \in \mathbb N.
\eeqn
Observe that $(m-k-1)+p \geqslant k+3 > k+1$. Thus, $S_{\lambdab}$ is eventually periodic, which is a contradiction. Hence, $U(W_{k+1}) = \tilde W_{k+1}$ and this completes the proof of the claim.

Since $U$ is a unitary, the identity $U(W_k) = \tilde W_k$ for all $k \in \mathbb N$, yields that 
\[\sum_{v \in V_\prec \cap \mathscr G_n}\Big(\text{Card}\big(\child v\big)-1\Big) = \sum_{\tilde v \in 
\tilde V_\prec \cap \tilde{\mathscr G}_n}\Big(\text{Card}\big(\child {\tilde v}\big)-1\Big),\ n \in \mathbb N.\]
Observe that if $v \in \mathscr G_n$ but $v \notin V_\prec$, then $\text{Card}\big(\child v\big)-1 = 0$. Thus, above becomes
\[\sum_{v \in \mathscr G_n}\Big(\text{Card}\big(\child v\big)-1\Big) = \sum_{\tilde v \in \tilde{\mathscr G}_n}\Big(\text{Card}\big(\child{\tilde v}\big)-1\Big),\ n \in \mathbb N.\]
This, in turn, gives that
\[\text{Card}(\mathscr G_{n+1}) - \text{Card}(\mathscr G_n) = \text{Card}(\tilde{\mathscr G}_{n+1}) - \text{Card}(\tilde{\mathscr G}_n),\ n \in \mathbb N.\]
Since $\text{Card}(\mathscr G_0) = \text{Card}(\tilde{\mathscr G}_0) = 1$, we obtain that
\[\text{Card}(\mathscr G_n) = \text{Card}(\tilde{\mathscr G}_n)\ \text{ for all }\ n \in \mathbb N.\]

Conversely, assume that $c_n = \tilde c_n$ and $\text{Card}(\mathscr G_n) = \text{Card}(\tilde{\mathscr G}_n) \text{ for all } n \in \mathbb N.$ 
Then $\dim W_n = \dim \tilde W_n$ for all $n \in \mathbb N$. Consider the linear map $U : E \to \tilde E$ given by $U = \oplus_{n \in \mathbb N}\, U_n$, where $U_n$ is a unitary from $W_n$ onto $\tilde W_n$. In view of \eqref{ker-W-n}, it is easy to see that $U$ is a unitary from $E$ onto $\tilde E$. Now we show that $U$ satisfies \eqref{unit-equiv}. To this end, let $f \in E$. Then by \eqref{ker-W-n}, we get
\[f = \sum_{m\in \mathbb N} f_m,\quad f_m \in W_m. \]
Using Lemma \ref{lemma-cor}, we get for all integer $n \geqslant 1$,
\beqn
U S_{\lambdab}^{*n} S_{\lambdab}^n f &=& \sum_{m\in \mathbb N} U S_{\lambdab}^{*n} S_{\lambdab}^n f_m = \sum_{m\in \mathbb N} c_{m}^2 \cdots c_{m+n-1}^2 U f_m\\ 
&=& \sum_{m\in \mathbb N} \tilde c_{m}^2 \cdots \tilde c_{m+n-1}^2 U f_m = \sum_{m\in \mathbb N} \tilde S_{\lambdab}^{*n} \tilde S_{\lambdab}^n U f_m = \tilde S_{\lambdab}^{*n} \tilde S_{\lambdab}^n Uf.
\eeqn
The second last inequality in above follows from the fact that $Uf_m \in \tilde W_m$ for all $m \in \mathbb N$. Therefore, by Lemma \ref{lem-3.5}, $S_{\lambdab}$ is unitarily equivalent to $\tilde S_{\lambdab}$. This completes the proof.
\end{proof}

\section{Concluding remarks}

The case of eventually periodic balanced weighted shifts warrants further attention. Note that the converse part of the proof of Theorem \ref{Thm3.5} does not assume that the weighted shifts are non-periodic. Hence, the criteria provided in Theorem \ref{Thm3.5} is sufficient but not necessary for the unitary equivalence of two eventually periodic balanced weighted shifts. For example, consider the simplest case of isometric weighted shifts $S_{\lambdab}$ and $\tilde S_{\lambdab}$ on the directed trees $\mathscr T$ and $\tilde{\mathscr T}$, respectively shown below: 
\begin{figure}[H]
\includegraphics[scale=.5]{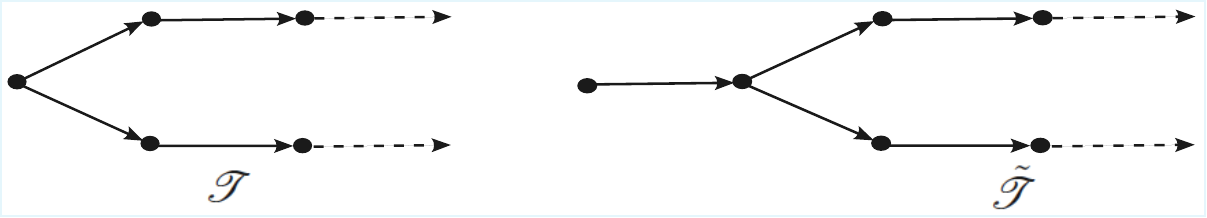}
\end{figure}
\noindent Then $c_n = \tilde c_n = 1$ for all $n \in \mathbb N$. Observe that $\dim \ker S_{\lambdab}^* = \dim \ker \tilde S_{\lambdab}^* = 2$. Hence by Wold decomposition, both $S_{\lambdab}$ and $\tilde S_{\lambdab}$ are unitarily equivalent. But $\text{Card}(\mathscr G_1) \ne \text{Card}(\tilde{\mathscr G}_1)$. Moreover, a bigger question about the unitary equivalence of two arbitrary weighted shifts on rooted directed trees remains open till date. 


\end{document}